\documentclass{amsart}

\newtheorem{theorem}{Theorem}[section]
\newtheorem{lemma}[theorem]{Lemma}

\theoremstyle{definition}
\newtheorem{definition}[theorem]{Definition}

\theoremstyle{remark}
\newtheorem{remark}[theorem]{Remark}

\numberwithin{equation}{section}

\begin{document}

\title{Nash's Existence Theorem for Non-compact Strategy Sets}

\author{Xinyu Zhang}
\address{Beijing Normal University-Hong Kong Baptist University United International College, Zhuhai 519087, P.R. China}
\email{zhangxinyu@uic.edu.cn}

\author{Cheng Chen}
\address{School of Mathematics, Sichuan University, Chengdu 610065, P.R. China}
\email{chenchengscu@outlook.com}
\thanks{This work was supported partially by NSF of China, No. 12071316. We would also like to thank Professor Zhang Shiqing from Sichuan University for his effective discussion.}

\author{Chunyan Yang*}
\address{School of Mathematics, Sichuan University, Chengdu 610065, P.R. China}
\email{yangchunyan0125@163.com}


\date{\today}



\begin{abstract}
This paper generalizes the Fan-Knaster-Kuratowski-Mazurkiewicz (FKKM) lemma to the case of weak topology, and obtains the Ky Fan minimax inequality defined on non-empty non-compact convex subsets in reflexive Banach spaces, then we apply it to game theory and obtain Nash's existence theorem for non-compact strategy sets, together with John von Neumann's existence theorem in two-player zero-sum games.
\end{abstract}

\maketitle



\section{Introduction}
In a non-cooperative game, consider $ n $ players named $ P_1, P_2, \cdots, P_n $, and each $ P_i $ has a set of strategies $ K_i $, here $ i = 1,2,\cdots,n $ and $ K_i $ satisfies the following condition:
\begin{enumerate}
	\item[(O)] Each $ K_i $ is a nonempty convex compact subset of a topological vector space $ E_i $.
\end{enumerate}

If each player $ P_i $ has chosen a strategy from $ K_i $, let
\[
f_i \colon X \to \mathbb{R}
\]
be the loss of player $ P_i $. Equivalently, $ -f_i $ gives $ P_i $'s payoff.

John Forbes Nash Jr., an American mathematician, introduced the following concept and proved its existence with a 28-page Ph.D. dissertation\cite{NashPhD} in 1950.
\begin{definition}[Nash]
	The equilibrium is a point
	\[
	\tilde{q} := (\tilde{q}_1,\tilde{q}_2,\cdots,\tilde{q}_n) \in X
	\]
	satisfying that
	\[
	f_i(\tilde{q}) = \min_{p_i \in K_i} f_i(\tilde{q}_1,\cdots,\tilde{q}_{i-1},p_i,\tilde{q}_{i+1}\cdots,\tilde{q}_n).
	\]
\end{definition}

In other words, the Nash equilibrium is a solution of a non-cooperative game that in it, no one can increase one's own expected payoff by changing one's strategy while the other players keep theirs unchanged.

John Nash's work on game theory shocked the economics community and earned him the John von Neumann Theory Prize in 1978 and the Nobel Memorial Prize in Economic Sciences (along with John Harsanyi and Reinhard Selten) in 1994.

In 1961, the Chinese-born American mathematician Ky Fan extended the classical Knaster-Kuratowski-Mazurkiewicz (KKM) lemma to an infinite-dimensional result \cite{FKKM}. Later in 1972, Ky Fan applied the FKKM lemma and obtained the Ky Fan minimax inequality \cite{FanKy} (which is equivalent to Brouwer's fixed-point theorem). The Ky Fan minimax inequality is a powerful tool and has many applications, especially in mathematical economics and game theory (see Chapter 9 in \cite{Aubin}). One of its applications is to show Nash's existence theorem in a very concise way (see Theorem 10.2.3 in \cite{ZhangShiqing}).

\section{Main Results and Proofs}
In this paper, we replace the condition (O) with the following one:
\begin{enumerate}
	\item[(H)] Each $ K_i $ is a nonempty convex subset of a reflexive Banach space $ E_i $.
\end{enumerate}
And we maintain the original definition of Nash equilibrium. Then, we denote the set of strategy profiles
\[
X := K_1 \times K_2 \times \cdots \times K_n \subseteq E_1 \times E_2 \times \cdots \times E_n =: E.
\]
For a vector $ p := (p_1,p_2,\cdots,p_n) \in E $, define its norm
\[
\lVert p \rVert_E := \sqrt{\sum_{i=1}^n \lVert p_i \rVert_{E_i}^2}.
\]
Since each $ E_i $ is a reflexive Banach space, it is not difficult to verify that $ E $ (equipped with its norm topology) is a reflexive Banach space either.

At first, we generalizes the FKKM lemma to the following one.
\begin{lemma}\label{newFKKM}
	Let $ X $ be a nonempty subset of a Banach space $ E $ and let the set-valued mapping
\[
T \colon X \to 2^E
\]
satisfy the following conditions:
\begin{enumerate}
	\item[(1)] For any fixed $ x \in X $, $ T(x) $ is a nonempty and weakly closed subset of $ E $.
	\item[(2)] There exists a $ x_0 \in X $ such that $ T(x_0) $ is weakly compact in $ E $.
	\item[(3)] For any finite set $ \{x_1,x_2,\cdots,x_n\} $, there holds
	\[
	\mathrm{co}\{x_1,x_2,\cdots,x_n\} \subseteq \bigcup_{i=1}^n T(x_i).
	\]
\end{enumerate}
Then
\[
\bigcap_{x \in X} \bigl(T(x)\cap T(x_0)\bigr) \neq \emptyset,
\]
and especially,
\[
\bigcap_{x \in X} T(x) \neq \emptyset.
\]
\end{lemma}

\begin{proof}
	Case 1. $ X $ is a set with finite points in it. Because of the equivalence of the norm topology and the weak topology on $ \mathrm{span} X $ which is a finite dimensional Banach space, in this case, the lemma is equivalent to the classic FKKM Lemma.
	
	Case 2. $ X $ is an infinite set. From hypotheses (1) and (2) we have that
	\[
	\tilde{T}(x) := T(x) \cap T(x_0)
	\]
	is weakly compact for any $ x \in X $.
	
	Next we will prove by contradiction. If the conclusion in this case is not true, we will show that there must exists such a finite set $ \{x_1,\cdots,x_m\} \subset X $ that
	\[
	\bigcap_{i = 1}^m \tilde{T}(x_i) = \emptyset,
	\]
	which is contradictory to Case 1.
	
	Assume that
	\[
	\bigcap_{x \in X} \tilde{T}(x) = \emptyset,
	\]
	we have
	\[
	\bigcup_{x \in X} \tilde{T}^c(x) = X
	\]
	by taking the complement of both sides. For an arbitrary point $ x_1 \in X $, the weakly compact set
	\[
	\tilde{T}(x_1) = X \setminus \tilde{T}^c(x_1) \subset \bigcup_{x \neq x_1} \tilde{T}^c(x).
	\]
	Notice that its right-hand side is a weakly open covering of $ \tilde{T}(x_1) $, So we can pick only a finite number of weakly open sets to cover $ \tilde{T}(x_1) $, i.e. there exists a finite set $ \{x_2,x_3,\cdots,x_m\} $ such that
	\[
	\bigcup_{i = 2}^m \tilde{T}^c(x_m) \supset \tilde{T}(x_1).
	\]
	So their complements satisfy
	\[
	\bigcap_{i = 2}^m \tilde{T}(x) \subset \tilde{T}^c(x_1).
	\]
	Hence we obtain that
	\[
	\bigcap_{i = 1}^m \tilde{T}(x) \subset \tilde{T}^c(x_1) \cap \tilde{T}(x_1) = \emptyset,
	\]
	which contradicts Case 1.
\end{proof}

Then we have the following generalized Ky Fan minimax inequality.
\begin{theorem}\label{newFanKy}
	Let $ X $ be a nonempty and convex subset of a reflexive Banach space $ E $ and let the functional
\[
f \colon X \times X \to \mathbb{R}
\]
satisfy the following conditions:
\begin{enumerate}
	\item[(i)] For any fixed $ y \in X $, the functional $ x \mapsto f(x,y) $ is quasi-concave on $ X $, i.e., for any $ l \in \mathbb{R} $, the set
	\[
	\{x \in X \vert f(x,y) \geq l\}
	\]
	is convex.
	\item[(ii)] For any fixed $ x \in X $, $ y \mapsto f(x,y) $ is weakly lower semicontinuous on $ X $, i.e., for any $ l \in \mathbb{R} $, the set
	\[
	\{(y,l) \in X \times \mathbb{R} \vert f(x,y) \leq l\}
	\]
	is weakly closed.
	\item[(iii)] $ m := \sup_{x \in X} f(x,x) < +\infty $.
	\item[(iv)] There exists a $ x_0 \in X $ such that the set
	\[
	T(x_0) := \{y \in X \vert f(x_0,y) \leq m\}
	\]
	is bounded in $ X $.
\end{enumerate}
Then
\[
\min_{y \in X} \sup_{x \in X} f(x,y) \leq \sup_{x \in X} f(x,x).
\]
\end{theorem}

\begin{remark}\label{coercive}
	Especially, the condition (iv) in Theorem \ref{newFanKy} is satisfied while the functional
	\[
	y \mapsto f(x_0,y)
	\]
	is coercive on $ X $, i.e. $ \lVert y_n \rVert_E \to +\infty $ implies $ f(x_0, y_n) \to +\infty $. If not, there must exist a sequence $ \{y_n\} \subset T(x_0) $ satisfying $ \lVert y_n \rVert_E \to +\infty $ so that $ f(x_0, y_n) \to +\infty $ which contradicts the definition of $ T(x_0) $. 
\end{remark}

\begin{proof}
	Set
\[
T(x) = \{y \in X \vert f(x,y) \leq m\}.
\]

(1) From hypotheses (ii) and (iii) we have that for any $ x \in X $, $ T(x) $ is nonempty (since $ x $ must belong to $ T(x) $) and weakly closed (by the definition of weak lower semicontinuity).

(2) By hypothesis (iv) we get the boundedness of $ T(x_0) $. Using the Eberlein-\v{S}mulian theorem (see Page 144 in \cite{Yosida}), $ T(x_0) $ becomes weakly relatively compact. Then, combined with conclusion (1) above, $ T(x_0) $ is weakly compact.

(3) There must holds
\[
\mathrm{co}\{x_1,x_2,\cdots,x_n\} \subseteq \bigcup_{i=1}^n T(x_i)
\]
for any finite set $ \{x_1,x_2,\cdots,x_n\} $. Otherwise, there has to be a $ \bar{y} \in \mathrm{co}\{x_1,x_2,\cdots,x_n\} $ but $ \bar{y} \notin \bigcup_{i=1}^n T(x_i) $. Then
\[
f(x_i,\bar{y}) > m, \: \forall i = 1,2,\cdots,n.
\]
That is,
\[
\exists \varepsilon > 0 \text{ such that } f(x_i,\bar{y}) \geq m + \varepsilon.
\]
From hypothesis (i) we have that the set $ \{x \in X \vert f(x,\bar{y}) \geq m + \varepsilon\} $ is convex. Since each $ x_i $ belongs to it, $ \bar{y} $ is also in it. Then
\[
f(\bar{y},\bar{y}) \geq m + \varepsilon > m,
\]
which is contradictory to the definition of $ m $.

Using Lemma \ref{newFKKM}, we have that
\[
\bigcap_{x \in X} T(x) \neq \emptyset,
\]
i.e., there exists a $ \tilde{y} \in T(x) $, that is,
\[
f(x,\tilde{y}) \leq m, \: \forall x \in X.
\]
Hence
\[
\inf_{y \in X} \sup_{x \in X} f(x,y) \leq \sup_{x \in X} f(x,\tilde{y}) \leq \sup_{x \in X} f(x,x).
\]
Next we will show that the notation ``$ \inf $'' above can be replaced by ``$ \min $''. Notice the hypothesis (ii), for any real number $ l $, the set $ \{(y,l) \in X \times \mathbb{R} \vert f(x,y) \leq l\} $ is weakly closed in $ E $. Then the set
\[
\bigcap_{x \in X} \{(y,l) \in X \times \mathbb{R} \vert f(x,y) \leq l\}
\]
is also weakly closed, which implies that the functional
\[
y \mapsto \sup_{x \in X} f(x,y)
\]
is weakly lower semicontinuous on $ X $. Moreover, the set
\[
\{y \in X \vert \sup_{x \in X} f(x,y) \leq m\} = \bigcap_{x \in X} T(x)
\]
is weakly closed in $ E $. Since
\[
M := \bigcap_{x \in X} T(x) \subseteq T(x_0) = \{y \in X \vert f(x_0,y) \leq m\},
\]
which implies that the set $ M $ is a weakly compact subset of $ T(x_0) $. Since weakly lower semicontinuous functionals always have the minimum on weakly compact sets, the functional
\[
y \mapsto \sup_{x \in X} f(x,y)
\]
is able to reach its infimum in $ M $. Then clearly we have that
\[
\inf_{y \in X} \sup_{x \in X} f(x,y) = \min_{y \in M} \sup_{x \in X} f(x,y) = \min_{y \in X} \sup_{x \in X} f(x,y).
\]
Hence we have
\[
\min_{y \in X} \sup_{x \in X} f(x,y) \leq \sup_{x \in X} f(x,x),
\]
which completes the proof.
\end{proof}

Finally we have the Nash's existence theorem for non-compact strategy sets.
\begin{theorem}\label{Nash}
	Assume hypothesis (H) and the following ones:
	\begin{enumerate}
		\item[(I)] Let $ p = (p_1,\cdots,p_n) $. For each $ i = 1,\cdots,n $, fix all the components $ p_j $ when $ j \neq i $ and the functional
		\[
		p_i \mapsto f_i(p)
		\]
		is convex on $ K_i $.
		\item[(II)] Each $ f_i $ is weakly continuous on $ X $.
		\item[(III)] There exists a $ p_0 = (p_{10},\cdots,p_{n0}) \in X $ such that $ \forall i \in \{1,\cdots,n\} $, the set
		\[
		T(p_0) := \bigl\{q \in X \vert \sum_{i=1}^n \bigl[f_i(q) - f_i(q_1,\cdots,q_{i-1},p_{i0},q_{i+1},\cdots,q_n)\bigr] \leq 0\bigr\}
		\]
		is bounded in $ E $.
	\end{enumerate}
	Then there is at least one Nash equilibrium in $ X $.
\end{theorem}

\begin{proof}
	Set
	\begin{equation*}
		\begin{split}
			f \colon X \times X &\to \mathbb{R}
			\\(p,q) &\mapsto \sum_{i = 1}^{n} \bigl[f_i(q) - f_i(q_1,\cdots,q_{i-1},p_i,q_{i+1},\cdots,q_n)\bigr].
		\end{split}
	\end{equation*}
	and
	\begin{equation*}
		\begin{split}
			T \colon X &\to 2^X
			\\p &\mapsto \{q \in X \vert f(p,q) \leq 0\}.
		\end{split}
	\end{equation*}
	
	(i) From hypothesis (I) we obtain that for any fixed $ q \in X $, the functional
	\[
	p \mapsto f(p,q)
	\]
	is concave on $ X $.
	
	(ii) By hypothesis (II) we get the weak continuity of $ f $.
	
	(iii) $ f(p,p) = 0 $ for all $ p \in X $.

	(iv) Observing hypothesis (III) we have that $ T(p_0) $ is bounded in $ E $ since for each $ q \in T(p_0) $, we have that
	\[
	\lVert q \rVert_E = \sqrt{\sum_{i=1}^n \lVert q_i \rVert_{E_i}^2} \leq \sum_{i=1}^n \lVert q_i \rVert_{E_i} < +\infty.
	\]
	
	Using Theorem \ref{newFanKy}, there exists a $ \tilde{q} \in X $ such that
	\[
	f(p,\tilde{q}) \leq 0, \: \forall p \in X.
	\]
	In particular, for any $ i = 1,2,\cdots,n $, we choose
	\[
	p^i = (\tilde{q}_1,\cdots,\tilde{q}_{i-1},p^i_i,\tilde{q}_{i+1}\cdots,\tilde{q}_n) \in X.
	\]
	Then
	\[
	f_i(\tilde{q}) - f_i(p^i) \leq 0, \: \forall i \in \{1,\cdots,n\},
	\]
	that is,
	\[
	f_i(\tilde{q}) \leq f_i(p^i) = f_i(\tilde{q}_1,\cdots,\tilde{q}_{i-1},p^i_i,\tilde{q}_{i+1}\cdots,\tilde{q}_n), \: \forall p^i_i \in K_i \text{ and } \forall i \in \{1,\cdots,n\}.
	\]
	It means that $ \tilde{q} $ is exactly the Nash equilibrium and the proof is complete.
\end{proof}

In Theorem \ref{Nash}, the condition (III) is usually difficult to satisfy for $ n \geq 3 $, however it holds for the following practical and concise situation for the John von Neumann's two-person zero-sum game in the unbounded strategy sets. Comparing with the result of Zeidler (see Theorem 2.G. in Page 76 and Proposition 1 in Page 80 of \cite{Zeidler}), our assumptions are weaker, especially here we did not need the strict convexity of the space $ X $.
\begin{theorem}
	In a two-player zero-sum game, denote $ P_1 $'s loss functional as
	\[
	f_1 \colon K_1 \times K_2 =: X \to \mathbb{R},
	\]
	where $ K_i $ (i.e., the strategies of $ P_i $) is a nonempty convex set in a suitable reflexive Banach space $ E_i $. Then clearly that $ f_2 = - f_1 $. Assume that
	\begin{enumerate}
		\item[(S1)] For any fixed $ p_2 \in K_2 $, the functional $ p_1 \mapsto f_1(p_1,p_2) $ is convex and lower semicontinuous on $ K_1 $.
		\item[(S2)] For any fixed $ p_1 \in K_1 $, the functional $ p_2 \mapsto f_1(p_1,p_2) $ is concave and upper semicontinuous on $ K_2 $.
		\item[(S3)] There exists a $ p_0 := (p_{10},p_{20}) \in X $ such that the functional $ \cdot \mapsto f_1(\cdot,p_{20}) $ is coercive on $ K_1 $ and $ \cdot \mapsto - f_1(p_{10},\cdot) $ is coercive on $ K_2 $.
	\end{enumerate}
	Then there is at least one Nash equilibrium in $ X $ and
	\[
	\min_{p_1 \in K_1} \max_{p_2 \in K_2} f_1(p_1,p_2) = \max_{p_2 \in K_2} \min_{p_1 \in K_1} f_1(p_1,p_2).
	\]
\end{theorem}

\begin{proof}
	Set
	\begin{equation*}
		\begin{split}
			f \colon X \times X &\to \mathbb{R}
			\\(p,q) &\mapsto - f_1(p_1,q_2) + f_1(q_1,p_2).
		\end{split}
	\end{equation*}
	
	(i) For any fixed $ q \in X $, $ p \mapsto f(p,q) $ is concave on $ X $ since $ \forall x, y \in X $ and $ \forall \lambda \in [0,1] $, by applying conditions (S1) and (S2), there holds
	\begin{equation*}
		\begin{split}
			f\bigl(\lambda x + (1 - \lambda)y, q\bigr) &= - f_1\bigl(\lambda x_1 + (1 - \lambda)y_1,q_2\bigr) + f_1\bigl(q_1,\lambda x_2 + (1 - \lambda)y_2\bigr)
			\\&\geq \lambda \bigl(- f_1(x_1, q_2) + f_1(q_1, x_2)\bigr) + (1 - \lambda) \bigl(- f_1(y_1, q_2) + f_1(q_1, y_2)\bigr)
			\\&= \lambda f(x,q) + (1-\lambda)f(y,q).
		\end{split}
	\end{equation*}
	
	(ii) Similar to the conclusion (i) above, the function $ q \mapsto f(p,q) $ is convex on $ X $ (and of course quasi-convex), then we have the set $ T(p) := \bigl\{q \in X \vert f(p,q) \leq m\bigr\} $ is convex for every $ m \in \mathbb{R} $ and for any $ p \in X $. And clearly that both $ q \mapsto f_1(q_1,p_2) $ and $ q \mapsto -f_1(p_1,q_2) $ are lower semicontinuous on $ X $, we have that $ T(p) $ is closed in $ X $. Since $ T(p) $ is both closed and convex, using Mazur's lemma (see Page 6 in \cite{Ekeland}), we obtain that $ T(p) $ is weakly closed. Hence the functional $ q \mapsto f(p,q) $ is weakly lower semicontinuous on $ X $ for any fixed $ p $ by definition.
	
	(iii) $ f(p,p) = 0 $ for all $ p \in X $.
	
	(iv) By applying condition (S3) we have that the functional
	\[
	q \mapsto f(p_0,q) = -f_1(p_{10},q_2) + f_1(q_1,p_{20})
	\]
	is coercive on $ X $ and the set $ \bigl\{q \in X \vert f(p_0, q) \leq 0\bigr\} $ is bounded in $ X $ straightly by Remark \ref{coercive}.
	
	Hence, using Theorem \ref{newFanKy} we have that there exists a $ \tilde{q} \in X $ such that
	\[
	f(p,\tilde{q}) \leq 0, \: \forall p \in X.
	\]
	Similar to Theorem \ref{Nash}, this $ \tilde{q} $ is exactly the Nash equilibrium and
	\[
	\min_{p_1 \in K_1} f_1(p_1, \tilde{q}_2) = f_1(\tilde{q}_1,\tilde{q}_2) = \max_{p_2 \in K_2} f_1(\tilde{q}_1,p_2),
	\]
	that is,
	\[
	\max_{p_2 \in K_2} \min_{p_1 \in K_1} f_1(p_1, p_2) = f_1(\tilde{q}_1,\tilde{q}_2) = \min_{p_1 \in K_1} \max_{p_2 \in K_2} f_1(p_1,p_2).
	\]
\end{proof}

\bibliographystyle{amsplain}
\bibliography{ams}

\providecommand{\bysame}{\leavevmode\hbox to3em{\hrulefill}\thinspace}
\providecommand{\MR}{\relax\ifhmode\unskip\space\fi MR }
\providecommand{\MRhref}[2]{%
  \href{http://www.ams.org/mathscinet-getitem?mr=#1}{#2}
}
\providecommand{\href}[2]{#2}
\begin{thebibliography}{1}

\bibitem{Aubin}
Jean-Pierre Aubin, \emph{Optima and {E}quilibria}, Springer-Verlag, Berlin,
  1998.

\bibitem{Ekeland}
Ivar Ekeland and Roger T\'{e}mam, \emph{Convex {A}nalysis and {V}ariational
  {P}roblems}, Society for Industrial and Applied Mathematics (SIAM),
  Philadelphia, PA, 1999.

\bibitem{FKKM}
Ky~Fan, \emph{A {G}eneralization of {T}ychonoff's {F}ixed {P}oint {T}heorem},
  Math. Ann. \textbf{142} (1961), 305--310.

\bibitem{FanKy}
\bysame, \emph{A {M}inimax {I}nequality and {A}pplications}, Inequalities (New
  {Y}ork-{L}ondon) (O.~Shisha, ed.), vol. III, Academic Press, 1972,
  pp.~103--113.

\bibitem{NashPhD}
John~F. Nash, Jr, \emph{Non-cooperative {G}ames},  (1950), (no paging), Thesis
  (Ph.D.)--Princeton University.

\bibitem{Yosida}
K\^{o}saku Yosida, \emph{Functional {A}nalysis}, Springer-Verlag, Berlin, 1995.

\bibitem{Zeidler}
Eberhard Zeidler, \emph{Applied {F}unctional {A}nalysis}, Springer-Verlag, New
  {Y}ork, 1995.

\bibitem{ZhangShiqing}
Shiqing Zhang, \emph{Functional {A}nalysis with {A}pplications(in {C}hinese)},
  Science Press, Beijing, 2018.

\end{thebibliography}

\end{document}